\documentclass[11pt,english]{article}
\usepackage[utf8]{inputenc} 
\usepackage[T1]{fontenc}
\usepackage[margin=1.83 cm,bottom=16mm,footskip=7mm,top=16mm]{geometry}

\usepackage{amsthm}
\usepackage{amsmath}
\usepackage{amssymb}
\usepackage{setspace}
\usepackage{mathtools}
\usepackage{graphicx}
\usepackage[hidelinks]{hyperref}
\usepackage{thm-restate}
\usepackage{cleveref}
\usepackage{graphicx}
\usepackage{enumitem}
\usepackage{makecell}
\usepackage{mdframed}
\setlist[itemize]{leftmargin=*} 
\usepackage{framed}
\usepackage{subcaption}

\usepackage{floatrow}

\AtBeginDocument{%
   \def\MR#1{}
}

\usepackage[square,sort,comma,numbers]{natbib}
\theoremstyle{plain}

\newtheorem*{thm*}{Theorem}
\newtheorem{thm}{Theorem}
\Crefname{thm}{Theorem}{Theorems}
\numberwithin{thm}{section}

\newtheorem*{lem*}{Lemma}
\newtheorem{lem}[thm]{Lemma}
\Crefname{lem}{Lemma}{Lemmas}

\newtheorem*{claim*}{Claim}
\newtheorem{claim}[thm]{Claim}
\crefname{claim}{Claim}{Claims}
\Crefname{claim}{Claim}{Claims}

\newtheorem*{defn*}{Definition}
\Crefname{defn*}{Definition}{Definitions}

\newtheorem{prop}[thm]{Proposition}
\Crefname{prop}{Proposition}{Propositions}

\newtheorem{cor}[thm]{Corollary}
\crefname{cor}{Corollary}{Corollaries}

\crefname{conj}{Conjecture}{Conjectures}

\Crefname{qn}{Question}{Questions}

\newtheorem{obs}[thm]{Observation}
\Crefname{obs}{Observation}{Observations}

\Crefname{ex}{Example}{Examples}

\theoremstyle{definition}

\Crefname{prob}{Problem}{Problems}

\newtheorem{defn}[thm]{Definition}
\Crefname{defn}{Definition}{Definitions}

\theoremstyle{remark}
\newtheorem*{rem}{Remark}

\renewenvironment{proof}[1][]{\begin{trivlist}
\item[\hspace{\labelsep}{\bf\noindent Proof#1.\/}] }{\qed\end{trivlist}}

\newenvironment{altproof}[1][]{\begin{trivlist}
\item[\hspace{\labelsep}{\bf\noindent Alternative proof of #1.\/}] }{\qed\end{trivlist}}

\newcommand{\remove}[1]{}

\newcommand{\ceil}[1]{
    \left \lceil #1 \right \rceil
}
\newcommand{\floor}[1]{
    \left \lfloor #1 \right \rfloor
}

\newcommand{\eps}{\varepsilon}

\newcommand{\cp}[2]{$\left(#1,#2\right)$-covering property}
\newcommand{\pcp}[2]{$#1$-partite $#2$-covering property}
\newcommand{\pr}{\mathbb{P}}

\DeclareMathOperator{\tc}{tc}
\DeclareMathOperator{\h}{h}
\DeclareMathOperator{\hp}{hp}
\DeclareMathOperator{\ch}{ch}
\DeclareMathOperator{\Bin}{Bin}
\expandafter\def\expandafter\normalsize\expandafter{%
    \normalsize
    \setlength\abovedisplayskip{4pt}
    \setlength\belowdisplayskip{4pt}
    \setlength\abovedisplayshortskip{4pt}
    \setlength\belowdisplayshortskip{4pt}
}

\newcommand{\subs}{\subseteq}
\newcommand{\Gnp}{\mathcal{G}(n,p)}
\newcommand{\Gtc}{G^{\tc}}
\newmdenv[
  topline=false,
  bottomline=false,
  skipabove=\topsep,
  skipbelow=\topsep
]{siderules}

\date{}
\begin{document}
\title{\vspace{-0.9cm}
 Covering graphs by monochromatic trees and Helly-type results for hypergraphs}

\author{Matija Buci\'c\thanks{Department of Mathematics, ETH, Z\"urich, Switzerland. Email: \href{mailto:matija.bucic@math.ethz.ch} {\nolinkurl{matija.bucic@math.ethz.ch}}.}\and
D\'aniel Kor\'andi\thanks{Institute of Mathematics, EPFL, Lausanne, Switzerland. Email:
\href{mailto:daniel.korandi@epfl.ch} {\nolinkurl{ daniel.korandi@epfl.ch}}. Research supported in part by SNSF grants 200020-162884 and 200021-175977.}\and 
Benny Sudakov\thanks{Department of Mathematics, ETH, Z\"urich, Switzerland. Email:
\href{mailto:benjamin.sudakov@math.ethz.ch} {\nolinkurl{benjamin.sudakov@math.ethz.ch}}.
Research supported in part by SNSF grant 200021-175573.}}

\maketitle
\vspace{-0.8cm}
\begin{abstract}
How many monochromatic paths, cycles or general trees does one need to cover all vertices of a given $r$-edge-coloured graph $G$?
These problems were introduced in the 1960s and were intensively studied by various researchers over the last 50 years. 
In this paper, we establish a connection between this problem and the following natural Helly-type question in hypergraphs.
Roughly speaking, this question asks for the maximum number of vertices needed to cover all the edges of a hypergraph $H$ if it is known that any collection of a few edges of $H$ has a small cover. We obtain quite accurate bounds for the hypergraph problem and use them to give some unexpected answers to several questions about covering graphs by monochromatic trees raised and studied by Bal and DeBiasio, Kohayakawa, Mota and Schacht, Lang and Lo, and Gir\~ao, Letzter and Sahasrabudhe.
\end{abstract}

\vspace{-0.5cm}
\section{Introduction}
\vspace{-0.1cm}
Given an $r$-edge-coloured graph $G$, how many monochromatic paths, cycles or general trees does one need to cover all vertices of $G$? The study of such problems has a very rich history going back to the 1960's when Gerencs\'er and Gy\'arf\'as \cite{gerencser67} showed that for any $2$-colouring of the edges of the complete graph, there are two monochromatic paths that cover all the vertices. Gy\'arf\'as \cite{gyarfas-covering-paths} later conjectured that the same is true for more colours, i.e.\ in any $r$-edge-colouring of $K_n$, there are $r$ monochromatic paths covering the vertex set. This conjecture was solved recently for $r=3$ by Pokrovskiy \cite{alexey}, but it is still open for all $r\ge 4$. The best known bound in general is that $O(r\log r)$ monochromatic paths suffice, and is due to Gy\'arf\'as, Ruszink\'o, S\'ark\"ozy and Szemer\'edi \cite{gyarfas08}, improving on results 
by Gy\'arf\'as \cite{gyarfas-covering-paths} and Erd\H{o}s, Gy\'arf\'as and Pyber \cite{pyber}. A similar type of question can be asked if we want to cover with cycles instead of paths. In fact, most results mentioned above also hold for disjoint cycles instead of paths. For further examples, generalisations and detailed history, we refer the reader to a recent survey by Gy\'arf\'as \cite{gyarfas-survey}.

We will study the problem of covering graphs using monochromatic connected components. Let us denote by $\tc_r(G)$ the minimum $m$ such that in any $r$-edge-colouring of $G$, there is a collection of $m$ monochromatic trees that cover the vertices of $G$. Since each connected graph contains a spanning tree, we may replace ``tree'' in this definition by ``connected subgraph'' or ``component'', which we do without further mention throughout the paper. The question of covering graphs with monochromatic components was first considered by Lov\'asz in 1975 \cite{lovasz-covers} and Ryser in 1970 \cite{henderson}, who conjectured that $\tc_r(K_n)=r-1$, or in other words, given any $r$-edge-colouring of $K_n$ we can cover its vertices using at most $r-1$ monochromatic components. It is easy to see that $\tc_r(K_n)\le r$ by fixing a vertex and taking the $r$ monochromatic components containing it in each of the colours. On the other hand, it is not hard \cite{bal18} to construct classes of graphs that only miss very few edges but admit no cover with a number of monochromatic components that is even bounded by a function of $r$.

Given a graph $G$ it is not clear how to determine $\tc_r(G)$ and in particular if it can be bounded by a function of $r$ only. In this paper we develop a framework which allows one to translate this question to a covering problem for hypergraphs. We illustrate the merits of this approach by obtaining answers to various well-studied problems in the area. The first set of these problems is about covering random graphs using monochromatic components.


\subsection{Covering random graphs}
A common theme in the combinatorics of recent years is to obtain sparse random analogues of extremal or Ramsey-type results. For some examples, see Conlon and Gowers \cite{conlon2016combinatorial} and Schacht \cite{schacht2016extremal} and references therein. With this in mind, Bal and DeBiasio \cite{bal18} initiated the study of covering random graphs by monochromatic components. Following this, Kor\'andi, Mousset, Nenadov, \v{S}kori\'c and Sudakov \cite{cycle-cover} and Lang and Lo \cite{lang} studied a version of this problem in which one uses cycles instead of components, and Bennett, DeBiasio, Dudek and English \cite{bennett2019large} looked at a related problem for random hypergraphs.

Here, we focus on the original problem of covering random graphs with monochromatic components considered by Bal and DeBiasio \cite{bal18}. They proved that the number of components needed becomes bounded when $p$ is somewhere between $\left(\frac{r\log n}{n}\right)^{1/r}$ and $\left(\frac{r\log n}{n}\right)^{1/(r+1)}$.
\begin{thm}[Bal, DeBiasio] \label{thm:baldebiasio}
Let $r$ be a positive integer. Then for $G\sim \Gnp$,
\begin{enumerate}[label=(\alph*)]
    \item \label{itm:bd-lb} if $p\ll \left(\frac{r\log n}{n}\right)^{1/r}$, then w.h.p.\ $\tc_r(G)\to\infty$, and
    \item \label{itm:bd-ub} if $p\gg \left(\frac{r\log n}{n}\right)^{1/(r+1)}$, then w.h.p.\ $\tc_r(G)\le r^2$.
\end{enumerate} 
\end{thm}
They also made the conjecture that w.h.p.\ $\tc_r(\Gnp) \le r$ when $p\gg \left(\frac{r\log n}{n}\right)^{1/r}$. This was subsequently proved by Kohayakawa, Mota and Schacht \cite{mota} for $r=2$.\footnote{In fact they solve the partitioning version of this problem.} On the other hand, \cite{mota} presents a far from trivial construction, due to Ebsen, Mota, and Schnitzer, showing that $\tc_r(\Gnp)\ge r+1$ for $p \ll \left(\frac{ r\log n}{n}\right)^{1/(r+1)}$, which disproves the conjecture for $r \ge 3$. 
Since this example forces just one additional component and only applies for slightly larger values of edge probability, it was still generally believed that the conjecture is close to being true. 
In fact, Kohayakawa, Mota and Schacht ask if $r$ components are enough to cover $\Gnp$ when $p$ is slightly larger than $\left(\frac{ r\log n}{n}\right)^{1/(r+1)}$.



We show that the answer to the above question is quite different from what was expected. We also obtain a good understanding of the behaviour of $\tc_r(\Gnp)$ throughout the probability range.
In particular, we find that $\tc_r(G)$ only becomes equal to $r$ when the density is exponentially larger than conjectured.
\begin{thm}\label{thm:trees-end}
Let $r$ be a positive integer. There are constants $c,C$ such that for $G \sim \Gnp$,
\begin{enumerate}[label=(\alph*)]
    \item \label{itm:tree-lb} if $p<\left(\frac{c \log n}{n}\right)^{\sqrt{r}/2^{r-2}}$, then w.h.p.\ $\tc_r(G)>r$, and
    \item \label{itm:tree-ub} if $p>\left(\frac{C \log n}{n}\right)^{1/2^{r}}$, then w.h.p.\ $\tc_r(G) \le r$.
\end{enumerate} 
\end{thm}
It is easy to see that $\tc_r(G) \ge r$ holds whenever $\alpha(G)\ge r$ (see e.g.\ \cite{bal18}). So the second part of the theorem actually implies that $\tc_r(\Gnp)=r$ for all larger values of $p$, so long as $\alpha(\Gnp) \ge r$.

Moreover, we show that near the threshold where $\tc_r(G)$ becomes bounded (and for quite some time after that), its value is not linear, as had been conjectured, but is of order $\Theta(r^2)$.

\begin{thm}\label{thm:trees-start}
Let $r$ be a positive integer, $d>1$ a constant and $G=\Gnp$. There are constants $c,C$ such that if $\left(\frac{C \log n}{n}\right)^{\frac{1}{r}} < p < \left(\frac{c \log n}{n}\right)^{\frac{1}{d(r+1)}}$ then w.h.p.\ $\tc_r(G)=\Theta(r^2)$ (the asymptotics depending on $r$ only).
\end{thm}
Note that \Cref{thm:baldebiasio,thm:trees-start} together establish a threshold of $\left(\frac{\log n}{n}\right)^{1/r}$ for the property of having $\tc_r$ bounded by a function of $r$. A slightly weaker understanding of this threshold also follows from the result of Kor\'andi et al.\ \cite{cycle-cover} that an $r$-edge-coloured $\Gnp$ can be covered with $O(r^8\log r)$ monochromatic cycles whenever $p>n^{-1/r+\eps}$. In fact, our upper bound on $\tc_r(G)$ in this regime borrows some ideas from \cite{cycle-cover}.

The lower bound in \Cref{thm:trees-start} answers a question of Lang and Lo \cite{lang}, who considered the problem of partitioning $\Gnp$ into cycles, and ask if $o(r^2)$ monochromatic cycles are enough for $p=\Omega\left(\left(\frac{\log n}{n}\right)^{1/r}\right)$. Our result shows that not only is the answer no, it is not even possible for larger values of $p$, even if we only need to \emph{cover} the vertices and are allowed to use trees instead of cycles.

The above two theorems describe the value of $\tc_r(\Gnp)$ when $p$ is quite small or quite large. We also obtain the following theorem, which tracks the behaviour of $\tc_r(\Gnp)$ in the range between.

\begin{thm}\label{thm:trees-middle}
Let $k>r\ge 2$ be integers, there exist constants $c,C$ such that given $G \sim \Gnp$ if $\left(\frac{C \log n}{n}\right)^{1/k}<p<\left(\frac{c \log n}{n}\right)^{1/(k+1)}$ then w.h.p.\ 
$\frac{r^2}{20\log k} \le \tc_r(G) \le \frac{16r^2 \log r}{\log k}$.
\end{thm}
In fact, we obtain slightly better bounds when $k$ approaches either extreme (i.e.\ when $k$ is linear or exponential in $r$). This connects the bounds of this theorem to the ones in \Cref{thm:trees-start,thm:trees-end}. For example, we obtain that if $k$ is exponential in $r$ then $\tc_r(G)=\Theta(r)$.


\vspace{-0.2cm}
\subsection{The connection to covering partite hypergraphs.}\label{subsec:connec}
\vspace{-0.1cm}
As mentioned before, our proofs of the above results rely on an interesting connection to a natural problem about hypergraph covers. Loosely speaking, the question asks how big a cover of a hypergraph $H$ can be if any subgraph of $H$ with few edges has a small cover. Here by a (vertex) cover of a hypergraph $H$, we mean a set of vertices that has a non-empty intersection with all edges of $H$. The minimum size of such a cover is called the cover number of $H$, and is denoted by $\tau(H)$.

In our case, we consider a variant for $r$-partite $r$-graphs ($r$-uniform hypergraphs). A \emph{transversal cover} is then defined as a cover containing exactly one vertex in each part of the $r$-partition.
\begin{defn*}
We say that an $r$-partite $r$-graph $H$ has the \emph{\pcp{r}{k}} if any subgraph of $H$ with at most $k$ edges has a transversal cover.
\end{defn*}
We define $\hp_r(k)$ to be the largest possible cover number of an $r$-partite hypergraph
satisfying the \pcp{r}{k} if such a maximum exists, and set $\hp_r(k)= \infty$ otherwise. The connection between the two problems is particularly striking in the case of random graphs where we get the following result:
\begin{thm}\label{thm:equivalence}
Let $k>r\ge 2$ be integers, and let $G\sim \Gnp$. There are constants $C,c>0$ such that:
\begin{enumerate}
    \item \label{itm:ub} If $np^k>C\log n$ then w.h.p.\ $\tc_r(G) \le \hp_r(k)$.
    \item \label{itm:lb} If $np^{k+1}<c \log n$ then w.h.p.\ $\tc_{r+1}(G) \ge \hp_{r}(k)+1$.
\end{enumerate}
\end{thm}
What this is saying is that for $\left(\frac{\log n}{n}\right)^{1/k} \ll p \ll \left(\frac{\log n}{n}\right)^{1/(k+1)}$ the value of $\tc_{r}(\Gnp)$ is essentially determined by $\hp_r(k)$. 

Let us point out that the $k=r+1$ (first non-trivial) case of estimating $\hp_r(k)$ is directly related to Ryser's famous conjecture \cite{henderson} which claims $\tau(H)\le (r-1)\nu(H)$ for $r$-partite $r$-graphs. Here the \emph{matching number} $\nu(H)$ is defined as the largest number of pairwise disjoint edges of $H$. Indeed, it is easy to check that the \pcp{r}{(r+1)} for $H$ is equivalent to $\nu(H)\le r$. So in the special case when $\nu(H)=r$, Ryser's conjecture is equivalent to $\hp_r(r+1)\le r(r-1)$.

\vspace{-0.2cm}
\subsection{Covering by components of different colours}
\vspace{-0.1cm}
The second setting in which our method works very well is for graphs of large minimum degree.

Covering such graphs using monochromatic components was first considered by Bal and DeBiasio \cite{bal18}. A particularly nice conjecture they raised is that any graph $G$ with minimum degree $\delta(G)\ge (1-1/2^r)n$ can be covered by monochromatic components of \emph{distinct} colours. They gave an example showing that if true, this conjecture is best possible. Gir\~ao, Letzter and Sahasrabudhe \cite{girao2017partitioning} proved the conjecture for $r\le 3$. Our methods enable us to completely resolve this conjecture for all values of $r$.
\begin{thm}\label{thm:min-deg}
Let $G$ be an $r$-coloured graph on $n$ vertices with  $\delta(G)\ge (1-1/2^r)n$. Then the vertices of $G$ can be covered by monochromatic components of distinct colours.
\end{thm}

%

\vspace{-0.2cm}
\subsection{Covering general hypergraphs}
\vspace{-0.1cm}
For our connection discussed in \Cref{subsec:connec}, we needed the considered hypergraphs to be partite. As it turns out, the behaviour does not change much if we drop this condition. This gives rise to the following, independently interesting and perhaps even more natural, question: What is the largest possible cover number of an $r$-graph in which any $k$ edges have a cover of size at most $\ell$? Is it even finite? We denote this maximum by $\h_r(k,\ell)$ if it exists, and write $\h_r(k,\ell)=\infty$ otherwise. The question of determining the parameter $\h_r(k,\ell)$ was raised by Erd\H{o}s, Hajnal and Tuza \cite{erdos-hajnal-tuza} more than 25 years ago and was later studied by Erd\H{o}s, Fon-Der-Flaass, Kostochka and Tuza; Fon-Der-Flaass, Kostochka and Woodall; Kostochka   \cite{erdos-flaass-kostochka-tuza, kostochka-flaass-woodall, kostochka}. They were all interested in a setting when $r$ is big and $\ell$ is fixed and small.  Here we develop an approach which gives good estimates on $\h_r(k,\ell)$ in general and in particular when $\ell=r$, which is the case relevant to our problem.  

Let us now define the covering property in this setting, as well.
\vspace{-0.1cm}
\begin{defn*}
We say that a hypergraph $H$ has the \emph{\cp{k}{\ell}} if any subgraph of $H$ with at most $k$ edges has a cover of size at most $\ell$. 
\end{defn*}
\vspace{-0.2cm}
Note that a hypergraph $H$ has the \cp{k}{1} if and only if any $k$ edges of $H$ have a non-empty intersection. This property is the central concept in the study of Helly families, dating all the way back to 1913: Helly's theorem \cite{helly} states that if in a collection $\mathcal F$ of convex sets in $\mathbb{R}^d$, any $d+1$ sets have a non-empty intersection, then all the sets in $\mathcal F$ intersect. 
There is a vast literature on studying what other families exhibit such Helly-type properties. For some classical examples from geometry, see \cite{danzer-helly}, for Helly-type results for hypergraphs, see \cite{furedi,bollobas-book,lehel-helly}.
For example, a folklore analogue for hypergraphs states that if any $r+1$ edges of an $r$-graph intersect, then all the edges have a non-empty intersection (see \cite{furedi}). With our notation, this corresponds to $\h_r(r+1,1)=1$.
Studying $\h_r(k,\ell)$ leads to a natural generalisation of this result. 

An interesting special case is to determine when $\h_r(k,\ell)=\ell$ holds, i.e.\ for what $k$ we know that if any $k$ edges of a hypergraph have a cover of size $\ell$, then the whole hypergraph has a cover of size $\ell$. As observed by F\"uredi \cite{furedi}, a result of Bollob\'as \cite{bollobas-set-pairs} (generalising earlier work of Erd\H{o}s, Hajnal and Moon \cite{erdos-hajnal-moon} for graphs), implies
that the answer to this question is $k\ge \binom{\ell+r}{r}$.

We obtain good bounds on $\h_r(k,\ell)$ in general, but for the sake of clarity, and because this case illustrates the general behaviour well, we only present our results for $r=\ell$ here, summarised in the next theorem. Our general bounds are presented in \Cref{sec:covering}.

\begin{thm}\label{thm:table}
The following table describes the behaviour of $\h_r(k,r)$ for fixed $r$ as $k$ varies. For arbitrary constants $c>1$ and $d>0$ we have
\vspace{-0.4cm}
\begin{center}
\resizebox{\textwidth}{!}{
\begin{tabular}{c|c|c|c|c|c|c}
    Range of $k$ & $[1,r]$ & $r+1$ & $ (r, cr]$ & $\left(r , e^{r/2}\right]$& $\left[e^{dr}, \infty\right)$ & $\left[\binom{2r}r,\infty\right)$ \\
    \hline
    Value of $\h_r(k,r)$ & $\infty$ & $r^2$ & $\Theta\left(r^2\right)$ & $\left[ \frac{r^2}{4\log k},\frac{16r^2 \log r}{\log k}\right)$& $\Theta(r)$ & $r$  
\end{tabular}
}

\end{center}

\end{thm}
We can actually prove slightly stronger bounds in the middle range, when $k$ is close to either extreme, see \Cref{sec:covering} for more details. 

Another reason why $r=\ell$ is a case of particular interest for us is its relation to the parameter $\hp_r(k)$, which we need for our connection result. Note that $\hp_r(k) \le \h_r(k,r)$ follows immediately from the definitions.  On the other hand, examples giving lower bounds for $\h_r(k,r)$ need not be $r$-partite, so we need to work a bit harder to obtain lower bounds for $\hp_r(k)$. Nevertheless, we can establish essentially the same lower bounds for $\hp_r(k)$ as the ones above for $\h_r(k,r)$ (see \Cref{thm:table2}).


\subsection{Organisation of the paper}
\vspace{-0.2cm}
In the next section we present some preliminary results and tools, mostly simple properties of random graphs. In \Cref{sec:direct-cover} we will prove the upper bound of \Cref{thm:trees-start}. In \Cref{sec:connection} we describe the connection between the problem of covering by monochromatic trees and hypergraph covering problems in more detail and deduce \Cref{thm:equivalence} from it. In \Cref{sec:covering} we show our results for general hypergraphs, in particular, \Cref{thm:table}. Finally, in \Cref{sec:r-partite} we show our bounds on $\hp_r(k)$ which, through our connection to monochromatic covering, yield \Cref{thm:trees-end,thm:trees-start,thm:trees-middle,thm:min-deg}. In \Cref{sec:conc-remarks} we give some open problems. We conclude the paper with an appendix in which we give an alternative perspective to the general hypergraph setting considered in \Cref{sec:covering}.

\section{Preliminaries}
Let us start with recalling a couple of simple definitions about (hyper)graphs. The set of vertices and edges of a (hyper)graph $G$ are denoted by $V(G)$ and $E(G)$, respectively. An $r$-graph or $r$-uniform hypergraph is a hypergraph with all edges of size $r$. Whenever we work with $r$-graphs, we tacitly assume $r\ge 2$. An $r$-graph $H$ is said to be $r$-partite if there is a partition of $V(G)=V_1 \cup \dots \cup V_r$ such that any $e \in E(H)$ satisfies $|e \cap V_i|=1$ for all $i$. Throughout the paper all colourings considered are \emph{edge}-colourings. Whenever we have an $r$-colouring, we label the colours by the first $r$ positive integers.  

We now present several properties of random graphs that play a role in our arguments. Let us begin with the following simple Chernoff-type bound on the tail of the binomial distribution.
\begin{lem}[{\cite[Appendix A]{alon-spencer}}]\label{lem:chernoff}
Let $X \sim \text{Bin}(n,p)$ be a binomial random variable. Then 
\[ \pr(X<np/2) \le e^{-np/8}.\]
\end{lem}
The following result is a simple consequence of union bound and the above lemma. For a proof, see e.g.\ \cite[Lemma 3.8]{cycle-cover}.
\begin{lem}\label{lem:edge-between-sets}
Let $p=p(n)>0$. The random graph $G \sim \mathcal{G}(n, p)$ satisfies the following property w.h.p.: for any two disjoint subsets $A,B\subseteq V(G)$ of size at least $\frac{10 \log n}{p}$, there is an edge between $A$ and $B$.
\end{lem}

The following property allows us to give a lower bound on the number of common neighbours of a small set that send at least one edge to another small set. For a set of vertices $A$, we denote the set of common neighbours of all vertices in $A$ by $N(A)$. (For convenience, we define $N(\emptyset)=V(G)$.)
\begin{lem}\label{lem:common-nbors}
Let $r\ge 1$ be an integer, $D\ge 1$ and $p=p(n)\ge \left(\frac{64Dr \log n}{n}\right)^{1/r}$. The random graph $G \sim \mathcal{G}(n, p)$ satisfies the following property w.h.p.: for any two disjoint subsets $A,B\subseteq V(G)$ such that $|A| \le D/p$ and $|B|=r-1$, there are at least $|A|\log n$ vertices in $N(B)\setminus A$ that have a neighbour in $A$.
\end{lem}
\begin{proof}
Let $A,B$ be two fixed disjoint subsets of $V(G)$ of sizes $|A|=k \le D/p$ and $|B|=r-1$. The number of vertices of $N(B)\setminus A$ that have a neighbour in $A$ follows the binomial distribution\\ $\Bin\left(n-k-r+1, p^{r-1}(1-(1-p)^k)\right)$. Notice that for large enough $n$ we have $n'=n-k-r+1>n/2$, while $pk\le D$ implies $p'=p^{r-1}(1-(1-p)^k)\ge p^{r-1}(1-e^{-pk})\ge \frac{p^rk}{2D}$ using the fact that $1-e^{-x}\ge \frac{x}{2D}$ for $0 \le x\le D$ and $1\le D$. Applying \Cref{lem:chernoff}, we obtain that the probability that there are fewer than $|A| \log n \le k \cdot \frac{p^r n}{64Dr} < \frac{n}{2}\cdot \frac{p^r k}{4D} <\frac{n'p'}{2}$ vertices in $N(B)\setminus A$ that have a neighbour in $A$ is at most
\[ e^{-n'p'/8}\le e^{-knp^r/(32D)}\le n^{-2kr}. \]

Applying the union bound over all $1 \le k \le D/p$ and the $\binom{n}{k}\binom{n-k}{r-1}$ possibilities for $A$ and $B$, we obtain that the probability that there are sets $A,B$ failing the conditions of the problem is at most
\[ \sum_{k=1}^{D/p}n^kn^{r-1}n^{-2kr} \to 0 \]
as $n \to \infty$, as claimed.
\end{proof}

Taking $D=1$ and a set $A$ of size one in the above lemma gives us the following corollary
\begin{cor}\label{cor:common-nbors}
Let $r\ge 1$ be an integer and $G \sim \Gnp$ for $p>\left(\frac{64r \log n}{n}\right)^{1/r}$. Then w.h.p.\ any $r$ vertices in $G$ have at least $\log n$ common neighbours.
\end{cor}

As usual, an independent set in $G$ is a subset of $V(G)$ where no two vertices are connected by an edge, and the independence number $\alpha(G)$ is the largest size of an independent set in $G$. The following lemma is the only property we are going to use when proving lower bounds on $\tc_r(\Gnp)$. It is a special case of Lemma 6.4 (ii) in \cite{bal18}.
\begin{lem}\label{lem:random-for-lb}
Let $m>k\ge 2$ be integers. There is a $c>0$ such that for $G\sim \Gnp$ with $p\le \left(\frac{c \log n}{n}\right)^{1/k},$ w.h.p.\ there is an independent set $S$ in $G$ of size $m$ such that no $k$ vertices in $S$ have a common neighbour in $G$.
\end{lem}

The following statement is a simple generalisation of the observation that $\tc_r(K_n)\le r$.
\begin{prop} \label{prop:indepbound}
Let $r$ be an integer and $G$ be a multigraph. Then $\tc_r(G)\le r\alpha(G)$.
\end{prop}
\begin{proof}
Let $I$ be a maximal independent set, and take all monochromatic components containing a vertex of $I$. This is at most $r\alpha(G)$ components, and as $I$ is maximal, every other vertex is adjacent to $I$ in some colour, hence covered by some component.
\end{proof}

We finish the section with results of a somewhat different flavour. The following theorem, due to Alon \cite{alon-pairs}, is a far-reaching generalisation of the Bollob\'as set pairs inequality \cite{bollobas-set-pairs} and several of its variants (see \cite{furedi} for more details).

\begin{thm}[Alon] \label{thm:alon-set-pairs}
Let $V_1, \dots, V_s$ be disjoint sets, $a_1, \dots, a_s$ and $b_1,\dots,b_s$ be positive integers, and let $A_1, \dots, A_m$ and $B_1, \dots, B_m$ be finite sets satisfying the following properties:
\begin{itemize}
    \item $|A_i \cap V_j| \le a_j$ and $|B_i \cap V_j| \le b_j$ for all $1 \le i \le m$ and $1 \le j \le s$,
    \item $A_i \cap B_i = \emptyset$ for all $1 \le i \le m$ and
    \item $A_i \cap B_j \neq \emptyset$ for all $1 \le i < j \le m$.
\end{itemize}
Then $m \le \prod_{i=1}^s \binom{a_i+b_i}{a_i}$.
\end{thm}

\begin{defn}
A hypergraph $H$ is said to be \emph{critical} if every proper subgraph of $H$ has a strictly smaller cover number.
\end{defn}

\begin{cor}[Bollob\'as \cite{bollobas-set-pairs}] \label{cor:critical}
A critical $r$-graph with cover number $t+1$ has at most $\binom{r+t}{t}$ edges.
\end{cor}
\begin{proof}
Let $H$ be a critical $r$-graph with cover number $t+1$, and denote the edges of $H$ by $A_1,\dots, A_m$. By the criticality of $H$, we know that if we remove any edge $A_i$, the remaining hypergraph admits a cover of size at most $t$. Denote this cover by $B_i$. Notice that $A_i \cap B_i = \emptyset$, as otherwise $B_i$ would be a cover of $H$ of size at most $t$, which is impossible. Furthermore, $A_i \cap B_j \neq \emptyset$ whenever $i \neq j$, because $B_j$ is a cover of $H\setminus A_j$, so in particular, it covers the edge $A_i$.

This means that the sets of vertices $A_i,B_i$ satisfy $|A_i|=r$, $|B_i|\le t$, $A_i \cap B_i= \emptyset$ and $A_i \cap B_j \neq \emptyset$ whenever $i \neq j$. Therefore, \Cref{thm:alon-set-pairs} with $s=1$, $V_1=V(H), a_1=r$ and $b_1=t$ implies $m \le \binom{r+t}{r}$.
\end{proof}

\section{Upper bound near the threshold}\label{sec:direct-cover}

We start by showing an upper bound on $\tc_r(\Gnp)$ just above the probability threshold where it becomes bounded (w.h.p.). This is the only regime in which we will not use the relation between the problem of covering with monochromatic trees and the hypergraph covering problem.

As mentioned in \Cref{thm:baldebiasio}, Bal and DeBiasio \cite{bal18} obtained an upper bound of $r^2$ when $p\gg \left(\frac{\log n}{n}\right)^{1/(r+1)}$. Their argument proceeds as follows. First, one can define the \emph{transitive closure multigraph} $\Gtc$ of $G$ on the same vertex set, where two vertices are connected by an edge of colour $i$ whenever they are in the same colour-$i$ component of $G$. Clearly, $G$ and $\Gtc$ have the same set of monochromatic components. Now when $G\sim \Gnp$ with $p\gg \left(\frac{\log n}{n}\right)^{1/(r+1)}$, it is not hard to check that $\alpha(\Gtc)\le r$. Indeed, in this probability range, any $r+1$ vertices of $G$ have a common neighbour $v$ w.h.p., so some two of them will be connected by a monochromatic path through $v$. \Cref{prop:indepbound} then yields $\tc_r(G)=\tc_r(\Gtc)\le r^2$.

In some sense, our upper bounds through the hypergraph covering problem are far-reaching generalisations of this argument. However, all of these arguments break down for $\left(\frac{\log n}{n}\right)^{1/r}\ll p \ll \left(\frac{\log n}{n}\right)^{1/(r+1)}$, when we only know that any $r$ vertices have a common neighbour. This issue was resolved in \cite{cycle-cover} by embedding ``cascades'' in $\Gnp$ to show that among any $4r-2$ vertices, some two are connected by monochromatic paths in a robust way. The argument needed there is quite technical and only works for $p>n^{-1/r+\eps}$. We present a much simpler adaptation of the ideas to prove $\alpha(\Gtc)\le 3r-2$ for the whole probability range. This will readily imply the upper bound on $\tc_r(\Gnp)$ for \Cref{thm:trees-start}.

\begin{thm} \label{thm:cascadebound}
Let $r$ be a positive integer and $G\sim \Gnp$. There is a $C>0$ such that if $p>\left(\frac{C \log n}{n}\right)^{1/r}$, then w.h.p.\ $\tc_r(G) \le (3r-2)r$.
\end{thm}
\vspace{-0.38cm}
\begin{proof}
Fix an $r$-edge-colouring of $G$. By \Cref{prop:indepbound}, it is enough to show that $\alpha(\Gtc)\le 3r-2$.

So suppose for contradiction that there is an independent set $S$ of $3r-1$ vertices. Notice first that no vertex of $G$ can send two edges of the same colour to $S$, because such edges would belong to the same monochromatic component of $G$, so $\Gtc$ would have an edge in $S$. Notice that this implies that every vertex of $G$ has at most $r$ neighbours in $S$. 

Let us say that a vertex set $A$ is \emph{rainbow to} $\{v_1,\dots,v_r\}$ if $A$ is contained in the colour-$i$ component of $v_i$ for every $i$. Now let $X\subs S$ be any set of size $2r-1$, and let $A$ be a largest vertex set in $\Gtc$ that is rainbow to some $r$-subset $\{v_1,\dots,v_r\}\subs X$. Note that $|A|\ge 1$ since by \Cref{cor:common-nbors} we know that any $r$ vertices in $S$ have a common neighbour, which can not send two edges of the same colour towards $S$. The heart of our proof is the following bootstrapping argument, which immediately gives $|A|\ge \frac{10 \log n}{p}$.
\begin{claim}
If $A_0$ is a set of size at most $\frac{20r!}{p}$ that is rainbow to some $r$-set in $X$, then there is a set of size at least $\frac{|A_0|\log n}{r!}$ that is rainbow to some other $r$-set in $X$.
\end{claim}
\vspace{-0.38cm}
\begin{proof}
Suppose $A_0$ is rainbow to $X'=\{w_1,\dots,w_r\}\subs X$, and let $w_{r+1},\dots,w_{2r-1}$ denote the vertices of $B=X \setminus X'$. Note that $A_0$ is disjoint from $B$ because otherwise some vertex of $B$ is in the colour-1 component of $w_1$, contradicting the independence of $X$ in $\Gtc$.

This means that we can apply \Cref{lem:common-nbors} with $D=20r!$ to $A_0$ and $B$ to find a set $A'$ of at least $|A_0|\log n$ common neighbours of $B$ that each have some neighbour in $A_0$. Take some vertex $u \in A'$. As we have noted before, all edges $uw_j$ for $j>r$ must have distinct colours. There are $r!$ possible assignments of $r-1$ distinct colours to the edges $uw_{r+1},\dots,uw_{2r-1}$, so there is a subset $A''\subs A'$ of at least $|A'|/r!$ vertices such that for every $j>r$, all edges between $A''$ and $w_j$ have the same colour. Let $c$ be the colour not used by these edges between $A''$ and $B$.

By definition, every vertex $u\in A''$ is adjacent to some $v\in A_0$. Let $c'$ be the colour of the edge $uv$. As $A_0$ is rainbow to $X'$, $u$ is in the colour-$c'$ component of $w_{c'}$. But $X$ is independent, so $c'$ cannot be the colour of any edge $uw_j$ for $j>r$. This means that $c'=c$ for every such $u$, and hence, $A''$ is rainbow to $\{w_c,w_{r+1},\dots,w_{2r-1}\}$ (possibly reordered), as needed.
\end{proof}

Let $Y= S \setminus \{v_1,\dots,v_r\}$ and let $B$ be a largest vertex set in $\Gtc$ that is rainbow to some $r$-subset $\{u_1,\dots,u_r\}\subs Y$. The above claim also shows that $|B|\ge \frac{10\log n}{p}$. Note also that $A$ and $B$ must be disjoint since otherwise $v_1,u_1\in S$ would belong to the same colour-1 component of $G$, contradicting the independence of $S$. This means that we can apply \Cref{lem:edge-between-sets} to get an edge $ab$ with $a\in A, b \in B$. If $ab$ has colour $c$, then this implies that $v_c$ and $u_c$ belong to the same colour-$c$ component, so $v_cu_c$ is an edge within $S$ in $\Gtc$, a contradiction.
\end{proof}

\section{The connection to hypergraph covering}\label{sec:connection}

In this section we establish a connection between covering graphs using monochromatic components and a covering problem for hypergraphs.

Given an $r$-edge-colouring $c$ of a graph $G$ we build the following auxiliary $r$-partite $r$-graph $H=H(G,c)$. The vertices of $H$ are taken to be the monochromatic components of $G$ (including singleton vertices). Note that we treat components of distinct colours as different vertices of $H$ even if they consist of the same set of vertices in $G$. For every $v \in V(G)$, we make an edge $m(v)$ in $H$ consisting of all the monochromatic components that $v$ belongs to. Note that each vertex belongs to exactly one component of each colour, so $H$ is $r$-uniform and $r$-partite, where the parts of $H$ are formed by the monochromatic components of the same colour. 

\begin{prop}\label{prop:restatement}
Given an edge-colouring $c$ of a graph $G$, the minimum number of monochromatic components needed to cover $V(G)$ equals $\tau(H(G,c))$. Moreover, if $H(G,c)$ has a transversal cover, then $V(G)$ can be covered using monochromatic components of distinct colours.
\end{prop}
\begin{proof}
Let $C_1,\dots, C_t$ be some monochromatic components that cover $V(G)$. For any edge $m(v)$ of $H$, we know that $v \in C_i$ for some $i$, so $C_i \in m(v)$. This implies that $C_1,\dots, C_t$ when viewed as vertices of $H=H(G,c)$ make a cover of $H$. On the other hand, if monochromatic components $C_1,\dots, C_t$ make a cover of $H$ then for any $v \in V(G)$, its corresponding edge $m(v)$ must contain some $C_i$. But this means that $v \in C_i$, and as $v$ was arbitrary, $C_1,\dots, C_t$ cover $V(G)$. Combining these observations implies the first part of the proposition. If $C_1,\dots,C_r$ make a transversal cover of $H$ then there is at most one $C_i$ in any part of the $r$-partition, or in other words at most one $C_i$ of any fixed colour. Since we know by the first part that $C_1 \cup \dots \cup C_r=V(G)$, the second part of the proposition follows.
\end{proof}

Given a colouring of $G$, this proposition allows us to translate the problem of determining the number of monochromatic components needed for covering $V(G)$ to determining the cover number of the auxiliary hypergraph. However, to determine $\tc_r(G)$, one needs to consider all colourings of $G$, and this translates to determining the minimum cover number among a class of $r$-partite $r$-graphs which arise from different colourings of $G$. As it turns out, the crucial parameter in determining $\tc_r(G)$ is the largest $k$ such that any $k$ vertices of $G$ have a common neighbour. This also works nicely with the auxiliary hypergraphs in the sense that if any $k$ vertices in $G$ have a common neighbour then for \textit{any colouring} of $G$, the auxiliary hypergraph will satisfy the \pcp{r}{k}. Recall that an $r$-partite $r$-graph has the \pcp{r}{k} if any $k$ of its edges have a transversal cover, and $\hp_r(k)$ is the maximum cover number of such a hypergraph. This will allow us to prove the following upper bound on $\tc_r(G)$. 
\begin{lem}\label{lem:common-nbors-relation}
Let $k> r\ge 2$ be integers, and let $G$ be a graph in which any $k$ vertices have a common neighbour.\footnote{For this argument, we consider a vertex to be adjacent to itself, so a common neighbour might be one of the $k$ vertices.} Then $\tc_r(G) \le \hp_r(k)$. Moreover, if we can find a transversal cover of any hypergraph having the \pcp{r}{k}, then $V(G)$ can be covered using components of distinct colours.
\end{lem}

\begin{proof}
Our goal is to show that in any $r$-colouring of $G$, we can cover the vertices of $G$ using at most $\hp_r(k)$ monochromatic components. So fix a colouring $c$ of $G$, and let $H=H(G,c)$ be the corresponding auxiliary $r$-partite $r$-graph.

Now take any set $S$ of $k$ edges in $H$. These edges correspond to $k$ vertices in $G$, so they have some common neighbour $v$. The key observation is that $m(v)$ is a transversal cover of $S$. Indeed, every hyperedge in $S$ is of the form $m(u)$ for some neighbour $u$ of $v$, and so the component of colour $c(uv)$ containing $u,v$ is in both $m(u)$ and $m(v)$. This means that $H$ satisfies the \pcp{r}{k}, so $\tau(H)\le \hp_r(k).$ By \Cref{prop:restatement} we can cover $V(G)$ using at most $\tau(H)\le \hp_r(k)$ monochromatic components, and if $H$ has a transversal cover we can do it using components of distinct colours.
\end{proof}

Combining \Cref{cor:common-nbors} and this lemma, we obtain part \ref{itm:ub} of \Cref{thm:equivalence}. It might be worth mentioning that the above proof gives an upper bound in terms of a slightly stronger \pcp{r}{k}, where the transversal cover of the $k$ edges is required to be another edge of the family. However, we are not aware of any improved bound on $\hp_r(k)$ using this extra condition. 

Somewhat unexpectedly it turns out that one can also lower bound $\tc_r(G)$ in terms of $\hp_r$. Here the relevant property of $G$ is that it contains a big enough independent set in which no $k+1$ vertices have a common neighbour. 

\begin{lem}\label{lem:lower-bound-equivalence}
Let $k>r \ge 2$ be integers. There is a $C=C(r)$ such that if $G$ is a graph with an independent set of size $C$ in which no $k+1$ vertices have a common neighbour in $G$ then $\tc_{r+1}(G) \ge \hp_{r}(k)+1$.
\end{lem}

Before turning to the proof let us give a definition of an auxiliary hypergraph and establish some of its properties, which we will use in the proof. Given an $r$-partite $r$-graph $H$ we construct an auxiliary $r$-coloured multigraph $\mathcal{A}(H)$. Let us denote by $P_1,\dots,P_r$ the parts of $H$. For the vertex set of $\mathcal{A}(H)$ we take $E(H)$ and for any $e \in E(H)$ we denote the corresponding vertex in $\mathcal{A}(H)$ by $a(e)$. As for the edge set, we place an edge of colour $i$ between two vertices of $\mathcal{A}(H)$ whenever the corresponding edges in $E(H)$ intersect in $P_i$. Note that if these edges intersect in several different parts, then we add multiple edges (which will be of distinct colours) between the same vertices. The following proposition establishes the properties of $\mathcal{A}(H)$ that we will use.

\begin{prop}\label{prop:cover-to-colour}
Let $H$ be an $r$-partite $r$-graph satisfying the \pcp{r}{k} and let $A=\mathcal{A}(H)$ be defined as above. Then
\begin{enumerate}
    \item We need at least $\tau(H)$ monochromatic components to cover $A$.
    \item We can partition any set $S$ of at most $k$ vertices in $A$ into $r$ parts $S_1,\dots,S_r$ (some possibly empty) such that for every $i$, $S_i$ is contained in a single component of colour $i$.
\end{enumerate}
\end{prop}

\begin{proof}
For part 1, if $a(e),a(f),a(g)$ are vertices of $A$ and $a(e)$ is joined to both $a(f)$ and $a(g)$ with an edge of same colour in $A$, say $i$, then $e,f$ and $g$ had to contain the same vertex in part $i$ of $H$. In particular, propagating this along any monochromatic path in $A$ we conclude that for all vertices of $A$ belonging to a monochromatic component $C$, their corresponding edges in $H$ need to contain the same vertex $v(C)$ (the one in the part corresponding to the colour of $C$). This in particular means that if we could cover all vertices of $A$ using $t<\tau(H)$ monochromatic components $C_1,\dots,C_t$ then $v(C_1),\dots, v(C_t)$ would give a cover of all edges of $H$ using less than $\tau(H)$ vertices, a contradiction.

For part 2, let $S=\{a(e_1),\dots, a(e_k)\}$ be a set of $k$ vertices in $A$. By the \pcp{r}{k}, there is a transversal cover of $\{e_1,\dots, e_k\}$ in $H$. This means that we can partition $\{e_1,\dots, e_k\}$ into $r$ sets such that all edges in the $i$th set contain the vertex of the transversal cover belonging to part $i$. Now let $S_i$ consist of the vertices $a(e_j)$ such that $e_j$ is in the $i$th set to obtain our desired partition of $S$.
\end{proof}

We are now ready to prove \Cref{lem:lower-bound-equivalence}.

\begin{proof}[ of \Cref{lem:lower-bound-equivalence}]
Consider an $r$-partite $r$-graph $H$ with the property that any $k$ edges of $H$ have a transversal cover and $\tau(H)= \hp_r(k)$. As we will see in \Cref{obs:trivial-partite}, parts \ref{itm:(2p)} and \ref{itm:(3p)}, $k>r$ implies that $\hp_r(k)\le r^2$, in particular, $\hp_r(k)$ is finite. By repeatedly removing edges from $H$, we may also assume that $H$ is critical, and hence, by \Cref{cor:critical}, has at most $\binom{r+r^2}{r}$ edges. Let us number the parts of $H$ by $[r]$, and let $m\le \binom{r+r^2}{r}$ be the number of edges in $H$. 

Let $C=\binom{r+r^2}{r}+1$ and let $G$ be a graph with an independent set $X$ of size $C$ such that no $k+1$ vertices in $X$ have a common neighbour. Let $S\subs X$ be a subset of size $m$. Since $C\ge m+1$ there exists a vertex $w$ that is not adjacent to $S$. Let us now take our auxiliary $r$-coloured graph $A=\mathcal{A}(H)$ and identify its $m$ vertices with the vertices of $S$.

Using this colouring we will define an $(r+1)$-edge-colouring of $G$ where the vertices cannot be covered with fewer than $\hp_r(k)+1$ monochromatic components.
Let $T= V \setminus S$, so $w\in T$. We colour all edges of $G$ induced by $T$ with colour $r+1$. Since $S$ is independent, all uncoloured edges are between $S$ and $T$. 
Since no $k+1$ vertices of $S$ have a common neighbour in $G$, any vertex $v \in T$ sends at most $k$ edges towards $S$. Let $S_v=N(v) \cap S$, since $|S_v| \le k$ by property 2 of ${A}$ we know we can split it into $r$ parts $S_1(v),\ldots, S_r(v)$ such that each $S_i(v)$ is contained in a monochromatic component of $A$ of colour $i$. We colour an edge from $v$ to $S$ in colour $i$ if its other endpoint belongs to $S_i(v)$. This completes our colouring.

Let us now show that we need $\hp_r(k)+1$ monochromatic components to cover the vertex set of $G$ in our colouring. Note first that all edges touching $w$ have colour $r+1$, so we need one component of this colour. However, this component is disjoint from $S$ because no edge touching $S$ has colour $r+1$. Note also that any monochromatic component in $G$, when restricted to $S$, is contained in a monochromatic component of ${A}$. This is because for any two vertices $u_0,u_t \in S$ from a monochromatic component of colour $j$ in $G$, there is a path $u_0u_1\dots u_t$ in $S$ and vertices $v_0,\dots, v_{t-1}\in T$ such that $u_{i-1},u_i \in S_j(v_i)$. By part 2 of \Cref{prop:cover-to-colour}, this means that $u_{i-1}$ and $u_i$ are in the same monochromatic component of colour $j$ in $A$, for every $i$. But then $u_0$ and $u_t$ belong to the same monochromatic component, as well. 

This shows that we need at least as many components to cover $S$ as we need to cover $\mathcal{A}(H)$, which, by part 1 of \Cref{prop:cover-to-colour}, is at least $\tau(H) = \hp_r(k)$. Since we also needed one extra component to cover $w$ we obtain $\tc_{r+1}(G) \ge \hp_r(k)+1$.
\end{proof}
This lemma, coupled with \Cref{lem:random-for-lb}, implies part \ref{itm:lb} of \Cref{thm:equivalence}.

\textbf{Remark. }Note that \Cref{cor:common-nbors} and \Cref{lem:random-for-lb} imply that for almost all graphs $G$ of certain fixed density the conditions of \Cref{lem:common-nbors-relation,lem:lower-bound-equivalence} hold with the same value of $k$ implying that $\tc_r(G)$ is essentially determined by $\hp_r(k).$ 

\section{Results for general hypergraphs}\label{sec:covering}
In this section we prove \Cref{thm:table}.
Recall that $\h_r(k,\ell)$ is the largest possible cover number of an $r$-graph $H$ satisfying the \cp{k}{\ell}, i.e.\ any $k$ edges of $H$ have a cover of size $\ell$.
Let us start with some easy observations, establishing some basic properties of $\h_r(k,\ell)$.
\begin{obs} \label{obs:trivial} For $r \ge 2$ we have:
\begin{enumerate}[label=(\arabic*),ref=(\arabic*)]
    \item \label{itm:(1)} $\h_r(\ell,\ell) = \infty$,
    \item \label{itm:(2)} $\h_r(\ell+1,\ell) \le r\ell$,
    \item \label{itm:(3)} $\h_r(k+1,\ell) \le \h_r(k,\ell)$ and
    \item \label{itm:(4)} $\ell \le \h_r(k,\ell)$.
\end{enumerate}
\end{obs}
\begin{proof}
For \ref{itm:(1)}, notice that any hypergraph satisfies the \cp{\ell}{\ell}. Indeed, given any $\ell$ edges, we can choose one vertex from each to obtain a cover of size at most $\ell$. So taking a hypergraph with arbitrarily large cover number implies the claim.

For \ref{itm:(2)}, let $H$ be an $r$-graph satisfying the \cp{\ell+1}{\ell}. Note that $H$ cannot contain $\ell+1$ pairwise disjoint edges, because covering them would require at least $\ell+1$ vertices. So any maximal set $S$ of pairwise disjoint edges in $H$ has size at most $\ell$. By its maximality, every edge of $H$ intersects some edge in $S$. But then the set of all vertices contained in the edges of $S$ form a cover of size at most $r\ell$.

For \ref{itm:(3)}, notice that the \cp{k+1}{\ell} implies the \cp{k}{\ell}, because if we can cover any $k+1$ edges with $\ell$ vertices, then the same holds for any $k$ edges, as well.

For \ref{itm:(4)}, we can take a hypergraph consisting of $\ell$ disjoint edges. This clearly satisfies the \cp{k}{\ell} for any $k$, and has cover number exactly $\ell$. 
\end{proof}

\subsection{Upper bounds}


Let us start with characterising when we have equality in \Cref{obs:trivial} part \ref{itm:(4)}. The following result was observed by F\"uredi \cite{furedi} (in a slightly different language), and is a simple consequence of \Cref{cor:critical} by Bollob\'as. We provide the proof for completeness.

\begin{thm}\label{thm:ubend}
If $k \ge \binom{r+\ell}{\ell}$ then 
\[ \h_r(k,\ell) = \ell. \]
\end{thm}
\begin{proof}
We know from \Cref{obs:trivial} part \ref{itm:(4)} that $\h_r(k,\ell) \ge \ell$, so it is enough to show the upper bound.
Suppose for contradiction that there is an $r$-graph with cover number at least $\ell+1$ satisfying that any $k$ of its edges have a cover of size at most $\ell$. Let $H$ be such a hypergraph with the smallest possible number of edges. Then $H$ is clearly critical, so we can apply \Cref{cor:critical} to see that it has at most $\binom{r+\ell}{r}$ edges. But as $k \ge \binom{r+\ell}{r}$, this means that $H$ admits a cover of size at most $\ell$, a contradiction. \end{proof}

Note that the bound on $k$ is tight: indeed, the complete $r$-graph on $r+\ell$ vertices satisfies the \cp{\binom{r+\ell}{\ell}-1}{\ell}, but cannot be covered with $\ell$ vertices.

\medskip
Obtaining any kind of improvement over the easy inequality $\h_r(k, \ell) \le r\ell$, for $k \ge \ell+1$, given by \Cref{obs:trivial} seems far from immediate. 
The following result obtains a good bound by combining a random sampling argument with the above ideas. 

\begin{thm} \label{thm:upper-bound-middle}
Given integers $t,k$ such that $\ell \le t \le r\ell$ and $k \ge \binom{r+t}{t}^{1/\lfloor t/\ell\rfloor}\cdot 2r\ell \log (r\ell)$ we have 
\[ \h_r(k, \ell)\le t. \]
\end{thm}
\begin{proof}

We proceed by contradiction. Let us assume that there is an $r$-uniform hypergraph that has no cover of size $t$, but in which any $k$ edges have a cover of size $\ell$. Let $H$ be such a hypergraph with minimum number of edges. Then $H$ is critical and \Cref{cor:critical} implies that it has at most $\binom{r+t}{t}$ edges.

Our strategy goes as follows. We start by showing that for any subgraph $G$ of $H$ there must be a set of $\ell$ vertices of $G$ covering many edges of $G$, as otherwise a random sample of $k$ edges in $G$ would not have any cover of size $\ell$. Armed with this claim, we take such a set of size $\ell$ covering many edges of $H$, remove it from the vertex set, and repeat the process with the remaining graph. The claim ensures that in each step we remove many edges, which in turn will imply that this process cannot repeat more than $\frac{t}{\ell}$ times before we reach an empty graph. This algorithm then yields $\frac{t}{\ell}\cdot\ell=t$ vertices covering $H$, which is a contradiction. 

Let us now make this argument precise. Let $m$ denote the number of edges in $H$ (so $m\le \binom{r+t}{t}$), and let $x=m^{-1/\lfloor t/\ell\rfloor}< 1$. This $x$ will be the proportion of edges that remain in $G$ after removing the $\ell$-set in a step of the procedure described above.
Given a hypergraph $G$ and a subset $S\subseteq V(G)$, we let $d_G(S)$ denote the number of edges in $G$ that share a vertex with $S$, i.e.\ the number of edges in $G$ covered by the set $S$\footnote{Note that $d_G(S)$ is not the codegree of $S$, which would count the number of sets \textit{containing} $S$.}. Let us start by showing the above-mentioned claim.
We denote by $e(G)=|E(G)|$.

\begin{claim}\label{claim:large-codegree}
For every non-empty subgraph $G$ of $H$, there is a subset $S\subseteq V(G)$ of size at most $\ell$ such that $d_G(S) > (1-x)e(G)$.
\end{claim} 
\begin{proof}[ of \Cref{claim:large-codegree}]
Let us assume, to the contrary, that there is a subgraph $G$ with the property that any subset of $\ell$ vertices in $G$ covers at most $(1-x)e(G)$ edges. We may assume that $G$ contains no isolated vertices, as otherwise we can just remove all of them without violating any of the properties. Then $|V(G)|\le r e(G) \le rm\le r \binom{r+t}{t}$. 

Let $T$ be a random set obtained by sampling $k$ (not necessarily distinct) edges of $G$, independently and uniformly at random. Then $T$ consists of at most $k$ edges of $H$, so it has a cover of size $\ell$. 

On the other hand, for any given set $S\subseteq V(G)$ of size $\ell$, the probability that it covers a single randomly sampled edge is $d_G(S)/e(G) \le 1-x$.  In particular, the probability that $S$ covers all of the $k$ independently sampled edges in $T$ is at most $(1-x)^k$. Finally, by the union bound, the probability that there is a set of size $\ell$ which covers all $k$ of the sets in $T$ is at most
\[ \binom{|V(G)|}{\ell}(1-x)^k \le (rm)^\ell e^{-xk} < 1, \]
where the last inequality uses $xk > \ell \log (rm)$, which follows from
\[ xk \ge m^{-1/\lfloor t/\ell\rfloor} \binom{r+t}{t}^{1/\lfloor t/\ell\rfloor} 2r\ell \log (r\ell) \ge  2r\ell \log (r\ell) = \ell \log ((r\ell)^{2r}) \]
and, using the assumption $t \le r \ell$,
\[ rm \le r\binom{2r\ell}{r}< (2r\ell)^r\le (r\ell)^{2r}. \]
But this means that there is a set of at most $k$ edges without a cover of size $\ell$, a contradiction.
\end{proof}

Let us define a sequence of subgraphs $G_{i+1} \subs G_{i} \subs \dots \subs G_1=H$, by taking a set $S_i$ of $\ell$ vertices of $G_i$ maximising $d_{G_i}(S_i)$ and setting $G_{i+1}=G_i - S_i$ for every $i\ge 1$. It is now easy to see that $|E(G_{i+1})| < x^i m$ for every $i\ge 1$. Indeed, this is clearly true if $G_i$ is empty, whereas otherwise we can apply \Cref{claim:large-codegree} to obtain $d_{G_i}(S_i) > (1-x)|E(G_i)|$ and $|E(G_{i+1})| < x |E(G_i)|$, and then use induction.

In particular, $|E(G_{\lfloor t / \ell\rfloor+1})|<x^{\lfloor t / \ell\rfloor}m=1$ means that $G_{\lfloor t / \ell\rfloor+1}$ is empty. But then the set $S_1 \cup \dots \cup S_{\lfloor t / \ell\rfloor}$ is a cover of $H$ of size at most $\lfloor t / \ell\rfloor\ell \le t$, which contradicts our assumption.
\end{proof}

The next statement inverts the inequality between $k$ and $t$ given by the above theorem to show the upper bounds of \Cref{thm:table} in the middle of the range.

\begin{cor}\label{cor:ubound-simple}
Let $e^r\ge k>r\ge 2$ and let $m=\frac{4r}{\log k}$. Then $\h_r(k,r) \le 4rm \log m$.
\end{cor}
\begin{proof}
Note that the bound is trivial from \Cref{obs:trivial} parts \ref{itm:(2)} and \ref{itm:(3)} when $4rm \log m \ge r^2$. We may therefore assume $4rm \log m < r^2$, which after plugging in the chosen value for $m$ and exponentiating is equivalent to $k(\log k)^{16} >(4r)^{16}$ which is in turn easily seen to imply $k \ge r^6$. Let $t=r \floor{4m \log m}$.  We only need to check that these values of $k,r,t$ satisfy the condition of \Cref{thm:upper-bound-middle} on $k$ with $\ell=r$.  To see this, observe that $4r^2 \log r \le r^4 < k^{2/3}$, and also,
\[ \binom{t+r}{r}^{1/\floor{4m \log m}}\le \left(\frac{e(t+r)}{r}\right)^{r/(3m \log m)}\le\left(e(1+4m \log m)\right)^{\log k/(12 \log m)} \le k^{1/3} \]
using $m \ge 4$ (from $k \le e^r$) in the last inequality.
\end{proof}
Note that this corollary, combined with \Cref{obs:trivial} part \ref{itm:(3)}, gives $\h_r(k,r)=O(r)$ when $k$ is exponential in $r$, and $\h_r(k,r)\le \frac{16r^2\log r}{\log k}$ in general for smaller $k$ (using \Cref{obs:trivial} part \ref{itm:(2)} for $\log k<4$) as claimed by \Cref{thm:table}. The remaining upper bounds of \Cref{thm:table} follow from \Cref{obs:trivial} parts \ref{itm:(1)}, \ref{itm:(2)} and \ref{itm:(3)} in the range $k\in [1,cr]$ and from \Cref{thm:ubend} in the range $k\in [\binom{2r}{r},\infty)$.

\subsection{Lower bounds}

Let us start with a simple result when $k=\ell+1$, which was already observed by Erd\H{o}s et al.\ in \cite{erdos-flaass-kostochka-tuza}.

\begin{prop}\label{lem:lb-start0}
$\h_r(\ell+1,\ell)=r\ell$.
\end{prop}
\begin{proof}
The upper bound $\h_r(\ell+1,\ell)\le r\ell$ is given by \Cref{obs:trivial} part \ref{itm:(2)}, so it is enough to construct an $r$-graph $H$ with cover number $r\ell$ where any $\ell+1$ edges can be covered with $\ell$ vertices.

We can actually choose $H$ to be the complete $r$-uniform hypergraph on $r\ell-1+r$ vertices. Indeed, $\tau(H)=r\ell$ because the complement of any $r\ell-1$ vertices induces an $r$-edge. On the other hand, $H$ has fewer than $r(\ell+1)$ vertices, so for any $(\ell+1)$-set of $r$-edges in $H$ there are two edges that intersect. We can therefore cover this set with $\ell$ vertices by taking a vertex in this intersection, and one vertex for each of the remaining $\ell-1$ edges.
\end{proof}

To give a lower bound on $\h_r(k,\ell)$ in general, we need to find an $r$-graph whose cover number is large, but for which any collection of $k$ edges has a cover of size $\ell$. The cover number is usually not difficult to estimate, but the \cp{k}{\ell} can be hard to grasp. As it turns out, a simple counting trick can give good estimates on the largest $k$ satisfying this property. (See also  \cite{erdos-flaass-kostochka-tuza} for a somewhat weaker statement).

\begin{thm}\label{thm:lb-middle}
Let $r\ge 2$ and $t\ge \ell$ be positive integers. For any $k< \binom{t+r}{\ell}/\binom{t}{\ell}$, we have $\h_r(k,\ell) > t$.
\end{thm}
\begin{proof}
Let $H$ be the complete $r$-uniform hypergraph on $t+r$ vertices. The cover number of $H$ is easily seen to be $t+1$, and we will show that this graph satisfies the \cp{k}{\ell} for any $k < \binom{t+r}{\ell}/\binom{t}{\ell}$. To see this, observe that every edge of $H$ is covered by all but $\binom{t}{\ell}$ vertex $\ell$-sets. But then for any $k$ edges, there are at most $k\binom{t}{\ell}<\binom{t+r}{\ell}$ vertex $\ell$-sets that do not cover all of them, which leaves at least one vertex cover of size $\ell$. 
\end{proof}

It will be convenient to formulate this trick in a more general form for future use in \Cref{sec:r-partite}. 

\begin{lem}\label{lem:lb}
Let $H$ be an $r$-uniform hypergraph and let $G$ be an $\ell$-uniform hypergraph on the same vertex set. Let $\delta$ be the minimum of $d_G(S)$ over $S\in E(H)$. Then any $\floor{\frac{e(G)-1}{e(G)-\delta}}$ edges of $H$ can be covered by an edge of $G$. 
\end{lem}

\begin{proof}
Let $k\le \frac{e(G)-1}{e(G)-\delta}$, and consider a collection of $k$ edges $S_1,\dots, S_k$ of $H$. There are at most $e(G)-\delta$ edges in $G$ disjoint from each $S_i$. This means that there are at most $k(e(G)-\delta)\le e(G)-1$ edges in $G$ disjoint from \emph{some} $S_i$. In particular, there is an edge of $G$ that intersects all of the $S_i$.
\end{proof}

This trick might seem to give weak bounds, but 
they are actually close to best possible in our applications. For example, we have seen that \Cref{thm:lb-middle} is tight for $t=\ell$. When $\ell=r$, it also implies the following result, which is only a $\log (\frac{r}{\log k})$ factor away from the upper bound of \Cref{cor:ubound-simple}.  We will discuss tightness in general in \Cref{sec:reduction}.

\begin{cor}\label{cor:lb-middle}
Let $r,k$ be integers such that $e^{r/2}>k>r \ge 2$. Then $\h_r(k,r) \ge \frac{r^2}{4 \log k}$. 
\end{cor}
\begin{proof}
Let $t=\floor{\frac{r^2}{2 \log k}}$. Then $t\ge r \ge 2$ and
$\binom{t+r}{r}/\binom{t}{r}\ge \left(\frac{t+r}{t}\right)^r> e^{r^2/(2t)}\ge k$,
so we can apply \Cref{thm:lb-middle} to obtain the result.
\end{proof}

Up to now, all our examples came from complete $r$-graphs. In the following lemma we show that when $k$ is very close to $\ell$ we obtain better lower bounds by taking multiple copies of complete graphs instead. 

\begin{lem}\label{lem:lb-start}
Let $r\ge 2$ and $k> \ell$ be positive integers. $\h_r(k,\ell)\ge \frac{r\ell^2}{6k}$.
\end{lem}
\begin{proof}
Let $H$ be the disjoint union of $\ceil{\ell/2}$ copies $H_1,\dots, H_{\ceil{\ell/2}}$ of the complete $r$-graph on $ \floor{\frac{r\ell}{3k}} +r$ vertices. It is easy to see that the cover number of each $H_i$ is at least $\frac{r\ell}{3k}$, and since there are $\ceil{\ell/2}$ copies, we have $\tau(H) \ge\frac{r\ell^2}{6k}$. We just need to prove that any $k$ edges can be covered with $\ell$ vertices.

Let us first show that any $\ceil{\frac{3k}{\ell}}$ edges in the same $H_i$ have a common vertex.
Indeed, every edge avoids exactly $\floor{\frac{r\ell}{3k}}$ vertices, so for any collection of $\ceil{\frac{3k}{\ell}}$ edges, there are at most $\ceil{\frac{3k}{\ell}} \cdot \floor{\frac{r\ell}{3k}}< r+\floor{\frac{r\ell}{3k}}$ vertices that are not contained in all of them. In particular, some vertex is contained in all of them.
When $\ell\le 2$, this already shows that $H$ satisfies the \cp{k}{\ell}, so we may assume $\ell>2$ from now on.

Now let $S$ be any collection of $k$ edges in $H$. It is easy to see that the edges of $S$ can be split into $t \le \ceil{\ell/2}+ k/\ceil{\frac{3k}{\ell}}$ subsets $S_1\cup \dots \cup S_t$ so that each $S_j$ consists of at most $\ceil{\frac{3k}{\ell}}$ edges, all from the same $H_i$. (With at most one $S_j$ containing fewer than $\ceil{\frac{3k}{\ell}}$ edges of $H_i$ for every $i$.) When $\ell>2$, we get $t\le \ell$, so taking a single vertex in the intersection of each $S_i$ gives a cover of $S$ with at most $\ell$ vertices.
\end{proof}

The lower bounds in \Cref{thm:table} now follow from \Cref{obs:trivial} parts \ref{itm:(1)} and \ref{itm:(3)} in the range $k\in[1,r]$, from \Cref{lem:lb-start0} for $k=r+1$, from  \Cref{lem:lb-start} in the range $k\in(r,cr]$, from \Cref{cor:lb-middle} in the range $k\in(r,e^{r/2}]$, and from \Cref{obs:trivial} part \ref{itm:(4)} in the remaining range $k \in (e^{r/2},\infty)$.

\section{Results for \texorpdfstring{$r$}{r}-partite hypergraphs}\label{sec:r-partite}

In this section we present our estimates on $\hp_r(k)$, the maximum cover number of an $r$-partite $r$-graph in which any $k$ edges have a transversal cover. Using the relationship with the monochromatic tree cover problem that we established in \Cref{sec:connection}, this will allow us to prove \Cref{thm:trees-end,thm:trees-start,thm:trees-middle,thm:min-deg} in \Cref{sec:proofs}.

The definitions clearly imply that $\hp_r(k)\le \h_r(k,r)$. Indeed, most of our upper bounds follow from the upper bounds on $\h_r(k,r)$ obtained in the previous section. On the other hand, we need new examples for the lower bounds, as the complete $r$-graphs or copies of complete $r$-graphs that we used there are not $r$-partite. Nevertheless, our lower bounds for $\hp_r(k)$ are only a small constant factor away from the bounds we obtained for $\h_r(k,r)$ in the previous section. We begin by stating an analogue of \Cref{thm:table} to summarise the results we are going to show in this section all in one place. 

\begin{thm}\label{thm:table2}
The following table describes the behaviour of $\hp_r(k)$ for fixed $r$ as $k$ varies. For arbitrary constant $c>1$ we have

\begin{center}
\resizebox{\textwidth}{!}{
\begin{tabular}{c|c|c|c|c|c}
    Range of $k$ & $[1,r]$ & $r+1$ & $ (r, cr]$ & $\left(r , e^{r}\right]$&  $\left[\binom{2r}r,\infty\right)$ \\
    \hline
    Value of $\hp_r(k)$ & $\infty$ & $[r(r-4),r^2]$ & $\Theta\left(r^2\right)$ & $\left[ \frac{r^2}{12\log k},\frac{16r^2 \log r}{\log k}\right)$& $r$  
\end{tabular}
}
\end{center}
\end{thm}

Let us proceed with some easy observations, akin to \Cref{obs:trivial}.
\begin{obs} \label{obs:trivial-partite} For $r \ge 2$ we have:
\begin{enumerate}[label=(\arabic*),ref=(\arabic*)]
    \item \label{itm:(1p)} $\hp_r(r) = \infty$,
    \item \label{itm:(2p)} $\hp_r(r+1) \le r^2$,
    \item \label{itm:(3p)} $\hp_r(k+1) \le \hp_r(k)$ and
    \item \label{itm:(4p)} $r \le \hp_r(k)$.
\end{enumerate}
\end{obs}
\begin{proof}
Part \ref{itm:(1p)} can be seen by taking an arbitrarily large collection of disjoint edges. Part \ref{itm:(2p)} follows from \Cref{obs:trivial} part \ref{itm:(2)}. Part \ref{itm:(3p)} follows, as before, because if any $k+1$ edges have a transversal cover, then the same is true for any $k$ edges, as well. Part \ref{itm:(4p)} can be shown by taking a hypergraph consisting of exactly $r$ disjoint edges. 
\end{proof}

\subsection{Upper bounds}

Since $\hp_r(k) \le \h_r(k,r),$ \Cref{thm:ubend} implies that $\hp_r(k)=r$ whenever $k \ge \binom{2r}{r}$. In this particular case, we can obtain a slightly better result:
\begin{thm}\label{thm:ubend-partite}
Let $r\ge 2$ be an integer. If $k\ge 2^r$, then any $r$-graph satisfying the \pcp{r}{k} has a transversal cover. In particular, $\hp_r(k)=r$.
\end{thm}
\begin{proof}
We follow a similar approach as in \Cref{thm:ubend}, but use the fact that the covers involved are transversals to obtain an improvement from \Cref{thm:alon-set-pairs}.

Suppose for contradiction that there is an $r$-uniform $r$-partite hypergraph without a transversal cover satisfying that any $2^r$ of its edges have a transversal cover. Let $H$ be such a hypergraph with the smallest possible number of edges. Let $A_1,\dots,A_t$ be the edges of $H$. By the minimality assumption, the $r$-graph $H\setminus A_i$ has a transversal cover $B_i$ for every $i$. Note that $A_i \cap B_i=\emptyset$, as otherwise $B_i$ would be a transversal cover of $H$, and $A_i \cap B_j \neq \emptyset$ whenever $i\neq j$, because $B_j$ is a cover of $H \setminus A_j$, which contains $A_i$. 

If $V_1,\dots, V_r$ denote the parts of the $r$-partition of $H$, then we have $|A_i \cap V_j|=1$ (as $H$ is $r$-uniform and $r$-partite), and $|B_i \cap V_j|=1$ (as $B_i$ is a transversal cover) for every $i$ and $j$. Therefore we can apply \Cref{thm:alon-set-pairs} with $s=r$ to get that $H$ has at most $\prod_{i=1}^r\binom{1+1}{1}=2^r$ edges. But then $H$ admits a transversal cover, a contradiction.
\end{proof}

This bound on $k$ is tight if we insist on the cover being transversal, as shown by the complete $r$-partite $r$-graph with 2 vertices in each part: this satisfies the \pcp{r}{(2^r-1)}, but does not have any transversal cover. 
As we will see later (see \Cref{lem:lb-2-copies}), \Cref{thm:ubend-partite} is very close to being tight, in terms of the bound on $k$, even if we allow any cover of size $r$. 

In general, \Cref{cor:ubound-simple} (or \Cref{thm:table}) gives the following bound using $\hp_r(k)\le \h_r(k,r)$.

\begin{thm} \label{thm:ubound-partite-simple}
For $e^r\ge k>r\ge 2$, we have $\hp_r(k)\le \frac{16r^2 \log r}{\log k}$.
\end{thm}


It might be interesting to mention that a slight improvement on this result can be obtained using an argument along the lines of \Cref{thm:ubend,thm:ubend-partite} that is based on a generalisation of \Cref{thm:alon-set-pairs} by Moshkovitz and Shapira \cite{generalisation-alon-set-pairs}. However, the improvement is very minor, so we omit the details.

\subsection{Lower bounds}

As we mentioned before, our lower bounds on $\h_r(k,r)$ do not carry over to $\hp_r(k)$, because complete $r$-graphs (our constructions) are not $r$-partite. In fact, it is not immediately clear how to find $r$-partite graphs that satisfy the \pcp{r}{k} and have large cover number.

One of the main difficulties is that the vertices of the smallest part cover all the edges, so to get a meaningful bound on the cover number, we need to construct an $r$-graph $G$, whose edges touch many vertices in each part. On the other hand, $G$ cannot contain any matching of more than $r$ edges, otherwise it fails the \pcp{r}{k}. This immediately rules out complete or random $r$-partite $r$-graphs as possible candidates for $G$ to improve \Cref{obs:trivial-partite} part \ref{itm:(4p)}. We present a construction that does significantly better.

For some $m$ and $t < r$, let $H_{r,t,m}$ be the following $r$-partite $r$-graph. In each part of the $r$-partition, we fix $m$ vertices that we will call \emph{important}. Now for every set $S$ of $r-t$ important vertices in different parts, we add an edge to $H_{r,t,m}$ containing $S$ and $t$ unique ``unimportant'' vertices in each of the remaining parts. 
(So $H_{r,t,m}$ has $\binom{r}{r-t}m^{r-t}$ edges, and each unimportant vertex belongs to exactly one edge.) For example, if $t=r$ then $H_{r,t,m}$ is empty and if $t=0$ it is the complete $r$-uniform, $r$-partite graph with $m$ vertices per part.

\begin{prop}\label{cla:cover-important}
$\tau(H_{r,t,m})=(t+1)m$
\end{prop}
\begin{proof}
By taking any $t+1$ parts and choosing all important vertices in these parts, we get $(t+1)m$ vertices that cover $H_{r,t,m}$. Indeed, each edge of $H_{r,t,m}$ contains an important vertex in $r-t$ parts, so it must contain one in a part we selected.

Now let $S$ be a set of at most $(t+1)m-1$ vertices. We will show it does not cover $H_{r,t,m}$. As long as there is an unimportant vertex $v\in S$, we can replace it with any important vertex in the unique edge containing $v$: all previously covered edges are still covered by $S$. So we may assume that $S$ only contains important vertices. However, since $|S|\le (t+1)m-1$, at least $r-t$ parts contain some important vertex that is not in $S$. Then the edge defined by these $r-t$ important vertices is not covered by $S$.
\end{proof}

Let us start with our bound for the end of the range (where $\hp_r(k)$ becomes $r$).
\begin{thm}\label{lem:lb-2-copies}
For any $k<\binom{r}{\floor{(r+1)/2}}+\binom{r}{\ceil{(r+1)/2}}$, we have $\hp_r(k)>r$.
\end{thm}

\begin{proof}
Let $H$ be the disjoint union of $H_1=H_{r,\floor{(r-1)/2},1}$ and $H_2=H_{r,\ceil{(r-1)/2},1}$. Then, by \Cref{cla:cover-important}, $\tau(H)=\ceil{\frac{r-1}{2}}+1+\floor{\frac{r-1}{2}}+1=r+1$. On the other hand, we will show that $H\setminus e$ has a transversal cover for any edge $e$. As $H$ has $\binom{r}{\floor{(r+1)/2}}+\binom{r}{\ceil{(r+1)/2}}$ edges, this will imply that it has the \pcp{r}{\left(\binom{r}{\floor{(r+1)/2}}+\binom{r}{\ceil{(r+1)/2}}-1\right)}.

So let $e$ be an edge in $H$, and assume that it belongs to $H_1$ (the case $e\in H_2$ is similar). Let $C_1$ be the set of $\floor{\frac{r-1}{2}}$ important vertices of $H_1$ that are not contained in $e$, and let $C_2$ consist of the $r-\floor{\frac{r-1}{2}}= \ceil{\frac{r-1}{2}}+1$ important vertices of $H_2$ in the parts not represented in $C_1$. It is easy to see that $C_1$ covers all edges in $H_1$ except $e$, and $C_2$ covers $H_2$. So $C_1 \cup C_2$ is indeed a transversal cover of $H\setminus e$.
\end{proof}


We now prove our general lower bound.
\begin{thm} \label{thm:lbound-partite-general}
For any $k>r\ge 2$, we have $\hp_r(k) \ge \frac{r^2}{12 \log k}$.
\end{thm}
\begin{proof}
We may assume that $\frac{r^2}{12 \log k}>r$, or equivalently, $e^{r/12}>k$, as otherwise the statement is trivial. Let $t=\floor{\frac{r-1}{2}}$ and $m=\floor{\frac{r+1}{4 \log k}}$. Since $e^{r/12}>k$ we have $\frac{r+1}{4 \log k}>2$, which implies $m=\floor{\frac{r+1}{4 \log k}} \ge \frac{2}{3}\cdot \frac{r+1}{4 \log k}$. Thus, by \Cref{cla:cover-important}, $H=H_{r,t,m}$ satisfies $\tau(H)=(t+1)m \ge \frac{r^2}{12 \log k}$. To prove that $H$ has the \pcp{r}{k}, we will apply \Cref{lem:lb} with $G$ chosen as the complete $r$-partite $r$-graph on the set of important vertices of $H$. Note that any edge of $G$ is transversal, so \Cref{lem:lb} implies that $H$ satisfies the \pcp{r}{\left(\frac{e(G)-1}{e(G)-\delta}\right)}, where $\delta$ is the minimum number of edges in $G$ intersecting an edge of $H$. We only need to check that $\frac{e(G)-1}{e(G)-\delta}\ge k$.

To see this, note that $G$ has $m^r$ edges, and every edge of $H$ is disjoint from exactly $(m-1)^{r-t}m^t$ of them. Thus indeed,
\[ \frac{e(G)-1}{e(G)-\delta} = \frac{m^r-1}{(m-1)^{r-t}m^t}> \left(1+\frac{1}{m}\right)^{r-t}>e^{(r-t)/(2m)}\ge k. \]
\end{proof}

\medskip

Let us now consider the beginning of the range. As mentioned in the introduction, the problem of determining $\hp_r(r+1)$ is a special case of Ryser's conjecture \cite{henderson}.  Improvements over the trivial upper bound of $r^2$ on $\tau(H)$ for $H$ with $\nu(H) \le r$ are only known for $r\le 5$ \cite{ryser-ub3,ryser-ub}, in which case they also provide an improvement over our upper bound \Cref{obs:trivial-partite} part \ref{itm:(2p)}. Tightness examples for Ryser's conjecture are not known for all values of $r$. In fact, they are only known to exist when $r-1$ is a prime power \cite{tuza-ryser}, $r-2$ is a prime \cite{ryser-p+2} or for certain special small values \cite{ryser-small,ryser-small-2,ryser-small-3}. There are, however, examples for every $r$ that are almost tight: Haxell and Scott \cite{penny-ryser} constructed intersecting $r$-partite $r$-graphs with cover number at least $r-4$ for all values of $r$. We can use them as a black box to get the following bound.
\begin{thm}\label{thm:lb-start-partite}
For every $r\ge 2$, we have $\hp_r(r+1) \ge r(r-4)$.
\end{thm}
\begin{proof}
Let $H'$ be an intersecting $r$-partite $r$-graph with cover number at least $r-4$, as given by \cite{penny-ryser}, and let $H$ consist of $r$ copies of $H'$. Then we clearly have $\tau(H)\ge r(r-4)$, and it is easy to see that $H$ also satisfies the \pcp{r}{(r+1)}. Indeed, among any $r+1$ edges, some two belong to the same copy of $H'$, and hence intersect. A vertex in the intersection can then be extended to a transversal cover by taking one vertex each from the remaining edges.
\end{proof}
Of course, the same argument gives $\hp_r(r+1)\ge r(r-1)$ whenever there is a construction for $r$ matching Ryser's conjecture.


\medskip
As in the previous section, taking more copies of our general construction is helpful at the beginning of the range. However, there is an added difficulty here in ensuring that the cover we get is transversal.

\begin{thm}\label{thm:start-lb}
For every $k$ and $r\ge 2$, we have $\hp_r(k) \ge \frac{r^3}{50k}$.
\end{thm}
\begin{proof}
If $\frac{r^3}{50k} \le r$ or $k \le r$, then the inequality follows from \Cref{obs:trivial-partite}. We may therefore assume $k>r>50$. Let $t=\floor{\frac{r^2}{10k}}$, and let $H$ consist of $\floor{\frac{r}{4}}$ disjoint copies of $H_{r,t,1}$. Then  \Cref{cla:cover-important} and $\frac{r}{4}>12$ imply $\tau(H)=(t+1)\floor{\frac{r}{4}}> \frac{r^3}{50k}$. Note that in a single copy of $H_{r,t,1}$, every part contains one important vertex, and every edge contains $r-t$ important vertices. This means that any set of at most $\frac{r}{2t}$ edges within the same copy contains at least $r-t \cdot \frac{r}{2t} = \frac{r}{2}$ common important vertices. 

Now let $S$ be any collection of $k$ edges in $H$. It is easy to see that the edges of $S$ can be split into $m \le \floor{\frac{r}{4}}+k/\floor{\frac{r}{2t}} \le \frac{r}{2}$ subsets $S_1\cup \dots \cup S_m$ so that each $S_i$ consists of at most $\frac{r}{2t}$ edges from the same copy. Now for every $i$, there are at least $\frac{r}{2}$ vertices common to all edges in $S_i$. But then we can greedily choose one vertex in a different part to cover each $S_i$. By adding an arbitrary vertex from any part where we have used no vertices so far, we obtain a transversal cover of our arbitrary collection of $k$ edges.
\end{proof}

\begin{rem}
Another class of examples which would give comparable lower bounds throughout this section is given as follows. The \emph{$(r,s)$-sum-hypergraph} $S_{r,s}$ is an $r$-partite hypergraph with $s+1$ vertices in each part $i$, denoted by $v_{i,j}$ for $0 \le j \le s$. The edge set is given by $ \{(v_{1,x_1},\dots, v_{r,x_r}) \mid \sum_{i=1}^r x_i=s\}$. It is not hard to see that the cover number of $S_{r,s}$ is $s+1$ when $r\ge 2$. Moreover, one can show, using \Cref{lem:lb}, that $S_{r,s}$ satisfies the \pcp{r}{k} with similarly good bounds on $k$ as obtained above.
\end{rem}


\subsection{Proofs of the main theorems} \label{sec:proofs}

\begin{proof}[ of \Cref{thm:trees-end}]
 \ref{itm:tree-lb}: Let $G \sim \Gnp$ with $p< \left ( \frac{c \log n}{n} \right)^{\sqrt{r}/2^{r-2}}$. Our goal is to show that w.h.p.\ $\tc_r(G) >r$. When $r\le 2$, our assumption implies $p \ll 1/n$. In this case, $\Gnp$ has more than 2 components w.h.p. (see e.g.\ \cite{alon-spencer}), so the statement follows. When $r\ge 3$, we have $\binom{r-1}{\floor{r/2}}+\binom{r-1}{\ceil{r/2}}>\frac{2^{r-2}}{\sqrt{r}}$, so \Cref{thm:equivalence}(2) implies $\tc_r(G)\ge \hp_{r-1}\left(\binom{r-1}{\floor{r/2}}+\binom{r-1}{\ceil{r/2}}-1\right)+1$ w.h.p.\ in this probability range.  Here $\hp_{r-1}\left(\binom{r-1}{\floor{r/2}}+\binom{r-1}{\ceil{r/2}}-1\right)>r-1$ follows from \Cref{lem:lb-2-copies}, so indeed, $\tc_r(G)>r$.

 \ref{itm:tree-ub}: Let $G \sim \Gnp$ with $p> \left ( \frac{C \log n}{n} \right)^{1/2^{r}}$. Our goal is to show that w.h.p.\ $tc_r(G) \le r$. \Cref{thm:equivalence}(1) implies that w.h.p.\ $\tc_r(G)\le \hp_r(2^k)$. The result now follows since \Cref{thm:ubend-partite} implies $\hp_r(2^k)\le r$.
\end{proof}

\begin{proof}[ of \Cref{thm:trees-start}]
Let $G \sim \Gnp$ with $\left ( \frac{C \log n}{n} \right)^{\frac{1}{r}}<p< \left ( \frac{c \log n}{n} \right)^{\frac{1}{d(r+1)}}$, where $d>1$. Our goal is to show that w.h.p.\ $\tc_r(G)= \Theta(r^2)$.
 \Cref{thm:cascadebound} shows that in this probability range w.h.p.\ we have $\tc_r(G)\le 3r^2$.
 On the other hand, for $r\ge 3$ \Cref{thm:equivalence}(2) shows that in this probability range w.h.p.\ we have $\tc_r(G)\ge \hp_{r-1}(d(r+1)-1)+1$. We can then apply \Cref{thm:start-lb} to get $\hp_{r-1}(d(r+1)-1)\ge \frac{r^2}{300d}$.
\end{proof}

\begin{proof}[ of \Cref{thm:trees-middle}]
Let $G \sim \Gnp$ with $\left ( \frac{C \log n}{n} \right)^{\frac{1}{k}}<p< \left ( \frac{c \log n}{n} \right)^{\frac{1}{k+1}}$, for $k>r \ge 2$. Our goal is to show that w.h.p.\ $\frac{r^2}{20\log k}\le \tc_r(G)\le \frac{16r^2\log r}{\log k}$. \Cref{thm:equivalence} shows that w.h.p.\ $ \hp_{r-1}(k)\le \tc_r(G)\le \hp_r(k)$. Then \Cref{thm:ubound-partite-simple} gives $\hp_r(k)\le \frac{16r^2\log r}{\log k}$. Also, for $r\ge 10$, \Cref{thm:lbound-partite-general} gives $\hp_{r-1}(k)\ge \frac{(r-1)^2}{12\log k}\ge \frac{r^2}{20\log k}$. For $r<10$, the lower bound is trivial.
\end{proof}

\begin{proof}[ of \Cref{thm:min-deg}]
Let $G$ be a graph with $\delta(G)\ge (1-1/2^r)n$. Our goal is to show that in any $r$-colouring of $G$ we can find a collection of trees of distinct colours which cover the vertex set of $G$. In order to apply  \Cref{lem:common-nbors-relation}, we need to show that any $k=2^r$ vertices $\{v_1,\dots,v_k\}$ have a common neighbour. Note that for this application, we can consider a vertex to be adjacent to itself, so if $M_i$ denotes the set of non-neighbours of $v_i$, then $|M_i|<n/k$. But then $|\bigcup_{i=1}^k M_i|<n$, so indeed there is a common neighbour (possibly one of the $v_i$). This means that we can apply \Cref{lem:common-nbors-relation} and then \Cref{thm:ubend-partite} to get $\tc_r(G)\le \hp_r(k)=r$. Moreover, \Cref{thm:ubend-partite} gives a transversal cover for the hypergraph, so according to  \Cref{lem:common-nbors-relation}, any $r$-edge-colouring of $G$ has a desired monochromatic tree cover using distinct colours.
\end{proof}

\section{Concluding remarks and open problems}\label{sec:conc-remarks}

In this paper we have obtained a very good understanding of how $\h_r(k,\ell)$ behaves as $k$ varies. However, there is a multiplicative gap of about $\log (\frac{\ell}{\log k})$ between our lower and upper bounds in the middle of the range. It would be interesting to determine the behaviour of $\h_r(k,\ell)$ more accurately. Any improvement on the upper bound would translate to improved bounds on $\hp_r(k)$ and $\tc_r$, as well. An improvement on the lower bounds, while still interesting, would only yield such an improvement if the example constructed was $r$-partite.

Another very interesting problem is to determine what is the smallest $k$ for which $\hp_r(k)=r$. The bounds given by \Cref{thm:ubend-partite,lem:lb-2-copies} ($2^r$ vs $2\binom{r}{r/2}$) are only a factor of $\sqrt{r}$ away from each other. We believe that \Cref{lem:lb-2-copies} is probably closer to the truth. In fact, it is tight for small values of $r$. 

Our \Cref{thm:trees-start} shows that the tree cover number of a random graph is $\Theta(r^2)$ just above the probability threshold where $\tc_r(\Gnp)$ becomes bounded. More precisely, we show that $r^2(1-o(1)) \le \tc_r(\Gnp) \le 3r^2$ when $p$ is slightly above $(\frac{\log n}{n})^{1/r}$, it would be interesting to know it more accurately.

Very recently, Kor\'andi, Lang, Letzter and Pokrovskiy \cite{cycle-threshold} obtained a related result about partitioning into monochromatic cycles with the same answer as in our \Cref{thm:trees-start}: They proved that $r$-edge-coloured graphs of minimum degree $\frac{n}{2}+\Theta(\log n)$ can be partitioned into $\Theta(r^2)$ disjoint monochromatic cycles.


\medskip
\textbf{Acknowledgements.} We thank Louis DeBiasio and Tuan Tran for helpful comments on an earlier version of this manuscript. We would like to thank the anonymous referees for their careful reading of the paper and many useful suggestions. In particular, we are grateful for a suggestion on how to rewrite the proof of \Cref{lem:lower-bound-equivalence} to make it easier to follow.

\providecommand{\bysame}{\leavevmode\hbox to3em{\hrulefill}\thinspace}
\providecommand{\MR}{\relax\ifhmode\unskip\space\fi MR }
\providecommand{\MRhref}[2]{%
  \href{http://www.ams.org/mathscinet-getitem?mr=#1}{#2}
}
\providecommand{\href}[2]{#2}



\appendix
\section{An alternative perspective} \label{sec:reduction}

In this section we present an alternative perspective for bounding $\h_r(k,\ell)$ compared to what we did in \Cref{sec:covering}. One advantage of this reformulation is that it can be used to show that \Cref{lem:lb} is close to optimal in our applications. We believe it could also be useful in improving the upper bounds on $\h_r(k,\ell)$ obtained in \Cref{thm:upper-bound-middle}.

Let $H$ be an $r$-graph. For a set $S \subseteq V(H)$, let $m(\bar{S})$ denote the set of edges in $H$ disjoint from $S$. We define the \emph{$\ell$-covering hypergraph} $\ch_\ell(H)$ of $H$ on vertex set $E(H)$, where for each subset $S\subseteq V(H)$ of size $\ell$, we take $m(\bar{S})$ as an edge of $\ch_\ell(H)$. Let us show how some properties of $\ch_\ell(H)$ relate to those of $H$. 

\begin{prop}\label{prop:relation}
Let $H$ be an $r$-graph.
\begin{enumerate}
    \item\label{itm:rel1} The smallest $k$ for which $H$ does \emph{not} satisfy the \cp{k}{\ell} is equal to $\tau(\ch_\ell(H))$.
    \item\label{itm:rel2} $\tau(H)> t\ell$ if and only if $\ch_\ell(H)$ is $t$-intersecting (i.e.\ any $t$ edges in $\ch_\ell(H)$ have a common vertex).
\end{enumerate}
\end{prop}
\begin{proof}
We start with part \ref{itm:rel1}.
Let $t=\tau(\ch_\ell(H))$ and let $T$ be a minimal cover of $\ch_\ell(H)$. Note that $T$ is a set of $t$ edges in $H$. This set cannot be covered with $\ell$ vertices in $H$, because if $S$ was such a cover, then $m(\bar{S})$ would be disjoint from $T$ in $\ch_\ell(H)$. Therefore, $H$ does not satisfy the \cp{t}{\ell}.

Let $T'$ be a set of at most $t-1$ edges of $H$. Since $T'$ is not a cover of $\ch_\ell(H)$, there is an edge $m(\bar{S})$ in $\ch_\ell(H)$ that is not covered by $T'$. This means that $S$ intersects every edge in $T'$, so $T'$ can be covered by $\ell$ vertices of $H$. Since $T'$ was arbitrary, $H$ satisfies the \cp{t-1}{\ell}.

Let us now turn to part \ref{itm:rel2}. If $H$ has a cover of size at most $t\ell$, then it can be covered with at most $t$ subsets of $V(H)$ of size $\ell$. The edges corresponding to these sets in $\ch_\ell(H)$ have no common vertex, so $\ch_\ell(H)$ is not $t$-intersecting.

Since $\ch_\ell(H)$ is not $t$-intersecting, $V(H)$ has $t$ subsets of size $\ell$ such that no edge of $H$ is disjoint to all of them. This means that union of these subsets covers $H$, and in particular, $\tau(H)\le t\ell$.
\end{proof}

This proposition allows us to translate the two concepts involved in $\h_r(k,\ell)$ into somewhat simpler parameters of covering hypergraphs. It will be useful for us to use an even simpler parameter, the relaxation of the covering number:

A fractional cover of a hypergraph $H$ is an assignment of weights to the vertices of $H$ so that each edge receives a total weight of at least 1. The minimum possible total weight assigned is called the \emph{fractional cover number}, and is denoted by $\tau^*(H)$. Lov\'asz \cite{lovasz-frac-cover} found the following connection between the two cover numbers when $H$ has maximum degree $d$:
\begin{align}\label{ineq:lovasz}
\tau^*(H) \le \tau(H) \le (1+ \log d) \tau^*(H).
\end{align}
As it turns out, several of our tools can be proved by combining \Cref{prop:relation} with \eqref{ineq:lovasz}.

\medskip
Indeed, we can now easily get a new proof of \Cref{lem:lb} (one of our main tools for bounding $\h_r(k,\ell)$ from below), one that actually shows that the lower bound we get on $k$ is close to optimal. Indeed, the edges of $G$ in \Cref{lem:lb} (as $\ell$-sets) define $e(G)$ edges in $\ch_\ell(H)$, and the condition of the lemma ensures that every vertex of $\ch_\ell(H)$ touches at most $e(G)-\delta$ of them (recall that $\delta$ is the minimum of $d_G(S)$ over $S \in E(H)$). So these edges form a subhypergraph $\mathcal{H}\subs \ch_\ell(H)$ of maximum degree at most $e(G)-\delta$. A double-counting argument then immediately implies $\tau^*(\ch_\ell(H)) \ge \tau^*(\mathcal{H}) \ge \frac{e(G)}{e(G)-\delta}$. Using the first inequality of \eqref{ineq:lovasz} and \Cref{prop:relation}, we get that $H$ has the \cp{k}{\ell} for any $k<\frac{e(G)}{e(G)-\delta}$, establishing \Cref{lem:lb}. Moreover, the same argument shows that $\tau^*(\ch_\ell(H))$ (which can easily be found with a linear program) approximates, up to a factor of $1+\log d(\ch_\ell(H))$ (here $d$ denotes the maximum degree as in \eqref{ineq:lovasz}), the smallest $k$ such that $H$ satisfies the \cp{k}{\ell}.

In \Cref{thm:lb-middle}, for example, we choose $H$ as the complete $r$-graph on $t+r$ vertices. As $\ch_\ell(H)$ is regular and uniform, it is not hard to see that $\tau^*(\ch_\ell(H))=\binom{t+r}{\ell}/\binom{t}{\ell}$ in this case. So \Cref{lem:lb} gives us the smallest $k$ for which this $H$ has the \cp{k}{\ell}, up to a factor of $1+\log \binom{|V(H)|}{\ell} \approx \ell \log |V(H)|$. A similar analysis gives essentially the same bound for the graph $H$ used in \Cref{thm:lbound-partite-general}, as well.

\medskip
\Cref{prop:relation} can also be used to obtain upper bounds on $\h_r(k,\ell)$. For example, in order to prove $\h_r(k,\ell)\le t\ell$, it is enough to show that if $\tau(\ch_\ell(H))=k+1$ for some hypergraph $H$, then $\ch_\ell(H)$ cannot be $t$-intersecting.
This raises the following natural question: What is the largest possible cover number of a $t$-intersecting hypergraph? 

A standard upper bound is $\frac{r-1}{t-1}+1$ for hypergraphs with hyperedges of size at most $r$ (see \cite{lovasz-book}), which is tight for certain complete $r$-graphs. However, it can be quite far from the truth for sparser hypergraphs with high uniformity. The following theorem works better in our case. 

\begin{thm}\label{thm:intersecting}
Let $G$ be a $t$-intersecting hypergraph with $n$ vertices and maximum degree $d$. Then $\tau(G)\le n^{1/t}(1+\log d)$.
\end{thm}
\begin{proof}
The proof is secretly the same as the one for \Cref{thm:upper-bound-middle}, but translated into the language of this setting. We do, however, benefit from this new setting by essentially replacing \Cref{claim:large-codegree} with \eqref{ineq:lovasz}.

\begin{claim*} Let $S\subseteq V(G)$ be a nonempty set. 
If $|e \cap S| \ge |S|/n^{1/t}$ for every edge $e \in E(G)$, then $\tau(G) \le (1+\log d(G))n^{1/t}$.
\end{claim*}
\begin{proof}
Let us assign weight $n^{1/t}/|S|$ to every vertex of $S$, and weight $0$ to every other vertex of $G$. This is a fractional cover because the total weight of every edge $e$ in $G$ is at least $|e \cap S|\cdot n^{1/t}/|S| \ge 1$. In particular, $\tau^*(G) \le n^{1/t}$, and using \eqref{ineq:lovasz} we get $\tau(G) \le (1+\log d(G))n^{1/t}$.
\end{proof}

We may now assume that for every nonempty subset $S \subseteq V(G)$, there is an edge $e \in E(G)$ such that $|e \cap S| < |S|/n^{1/t}$, as otherwise we are done by the claim. Let us denote this edge by $m(S)$. We construct a series of nested subsets of $V(G)$ as follows. We set $S_1=V(G)$, and define $e_i=m(S_i)$ and $S_{i+1}=S_i \cap e_i$ for $i=1,\dots,t$. This is possible because $S_{i+1}=e_1 \cap \dots \cap e_{i}$ and $H$ is $t$-intersecting. Moreover, it is easy to see by induction that $|S_{i+1}| < n^{1-i/t}$. But then $1 \le |e_1 \cap \dots \cap e_t|=|S_{t+1}|<1$ gives us a contradiction.
\end{proof}

Finally, let us show how \Cref{thm:upper-bound-middle} can be deduced from \Cref{thm:intersecting}.

\begin{altproof}[\Cref{thm:upper-bound-middle}]
As in the original proof, let us assume for contradiction that there is an $r$-uniform hypergraph $H$ satisfying the \cp{k}{\ell} with $\tau(H)>t$. As before, we may assume that $H$ is critical and has no isolated vertices, and thus has $|E(H)|\le \binom{r+t}{t}$ and $|V(H)|\le r\binom{r+t}{t}$. Note that for $G=\ch_\ell(H)$, we have $n=|V(G)|=|E(H)|\le \binom{r+t}{t}$ and since each edge of $G$ corresponds to an $\ell$-subset of $V(H)$ we get $d(G) \le \binom{|V(H)|}{\ell}\le |V(H)|^\ell \le r^\ell\binom{r+t}{r}^\ell \le (r+t)^{r\ell}\le (r+r\ell)^{r\ell} \le e^{2r\ell \log (r\ell)-1}$  using the assumption $t \le r \ell$. As  $\tau(H)>t$, we can apply \Cref{prop:relation} part \ref{itm:rel2} to deduce that $G$ is $\floor{t/\ell}$-intersecting. Thus \Cref{thm:intersecting} implies that $\tau(G) \le n^{1/\floor{t/\ell}} (1+\log d(G)) \le \binom{r+t}{t}^{1/\floor{t/\ell}}2r\ell \log (r\ell)$. Finally, \Cref{prop:relation} part \ref{itm:rel1} implies $k<\tau(G) \le \binom{r+t}{t}^{1/\floor{t/\ell}}2r\ell \log (r\ell)$, which is a contradiction.
\end{altproof}

\end{document}